\documentclass{amsart}
\usepackage{hyperref}

\newtheorem{theorem}{Theorem}[section]
\newtheorem{lemma}[theorem]{Lemma}
\newtheorem{corollary}[theorem]{Corollary}
\newtheorem{remark}[theorem]{Remark}

\newcommand{\R}{{\mathbb R}}
\newcommand{\N}{{\mathbb N}}
\newcommand{\Z}{{\mathbb Z}}
\newcommand{\C}{{\mathbb C}}

\newcommand{\si}{\sigma}
\newcommand{\siul}{\sigma^{\mathrm{u,l}}}
\newcommand{\siu}{\sigma^{\mathrm{u}}}
\newcommand{\sil}{\sigma^{\mathrm{l}}}
\newcommand{\sipmu}{\sigma_\pm^{\mathrm{u}}}
\newcommand{\sipml}{\sigma_\pm^{\mathrm{l}}}
\newcommand{\sipmul}{\sigma_\pm^{\mathrm{u,l}}}
\newcommand{\simpul}{\sigma_\mp^{\mathrm{u,l}}}
\newcommand{\lau}{\lambda^{\mathrm{u}}}
\newcommand{\lal}{\lambda^{\mathrm{l}}}
\newcommand{\laul}{\lambda^{\mathrm{u,l}}}
\newcommand{\la}{\lambda}
\newcommand{\pa}{\partial}
\newcommand{\nn}{\nonumber}
\newcommand{\be}{\begin{equation}}
\newcommand{\beq}{\begin{equation}}
\newcommand{\ee}{\end{equation}}
\newcommand{\eeq}{\end{equation}}
\newcommand{\bea}{\begin{eqnarray}}
\newcommand{\eea}{\end{eqnarray}}

\newcommand{\ol}{\overline}

\newcommand{\sign}{\mathop{\rm sign}}
\newcommand{\ov}{\overline}

\newcommand{\I}{\mathrm{i}}

\newcommand{\re}{\mathrm{Re}}
\newcommand{\im}{\mathrm{Im}}

\newcommand{\res}{\mathrm{Res}}

\newcommand{\clos}{\mathop{\rm clos}}
\newcommand{\inte}{\mathop{\rm int}}

\newcommand{\vprod}[2]{\!\!\!\!\begin{array}{c} \mbox{\raisebox{-0.5ex}[0.5ex]
{$\scriptstyle #2 $}} \\ \displaystyle \hspace*{1.1ex}\prod{}^* \\
\mbox{\raisebox{0.6ex}[-0.6ex]{$ \scriptstyle  #1 $}} \end{array}}
\newcommand{\vsum}[2]{\!\!\!\!\begin{array}{c} \mbox{\raisebox{-0.5ex}[0.5ex]
{$\scriptstyle #2 $}} \\ \displaystyle \hspace*{1.1ex}\sum{}^* \\
\mbox{\raisebox{0.6ex}[-0.6ex]{$ \scriptstyle  #1 $}} \end{array}}

\def\Xint#1{\mathchoice
   {\XXint\displaystyle\textstyle{#1}}%
   {\XXint\textstyle\scriptstyle{#1}}%
   {\XXint\scriptstyle\scriptscriptstyle{#1}}%
   {\XXint\scriptscriptstyle\scriptscriptstyle{#1}}%
   \!\int}
\def\XXint#1#2#3{{\setbox0=\hbox{$#1{#2#3}{\int}$}
     \vcenter{\hbox{$#2#3$}}\kern-.5\wd0}}
\def\dashint{\Xint-}

\newcommand{\sig}{\sigma}
\newcommand{\lam}{\lambda}
\newcommand{\ga}{\gamma}

\newcommand{\de}{\delta}

\numberwithin{equation}{section}

\begin{document}

\title[Scattering Theory for Steplike Quasi-Periodic Background]{Scattering Theory for
Jacobi Operators with General Steplike Quasi-Periodic Background}

\author[I. Egorova]{Iryna Egorova}
\address{Institute for Low Temperature Physics\\ 47,Lenin ave\\ 61164 Kharkiv\\ Ukraine}
\email{\href{mailto:egorova@ilt.kharkov.ua}{egorova@ilt.kharkov.ua}}

\author[J. Michor]{Johanna Michor}
\address{Imperial College\\
180 Queen's Gate\\ London SW7 2BZ\\ and International Erwin Schr\"odinger
Institute for Mathematical Physics, Boltzmanngasse 9\\ 1090 Wien\\ Austria}
\email{\href{mailto:Johanna.Michor@esi.ac.at}{Johanna.Michor@esi.ac.at}}
\urladdr{\href{http://www.mat.univie.ac.at/~jmichor/}{http://www.mat.univie.ac.at/\~{}jmichor/}}

\author[G. Teschl]{Gerald Teschl}
\address{Faculty of Mathematics\\
Nordbergstrasse 15\\ 1090 Wien\\ Austria\\ and International Erwin Schr\"odinger
Institute for Mathematical Physics, Boltzmanngasse 9\\ 1090 Wien\\ Austria}
\email{\href{mailto:Gerald.Teschl@univie.ac.at}{Gerald.Teschl@univie.ac.at}}
\urladdr{\href{http://www.mat.univie.ac.at/~gerald/}{http://www.mat.univie.ac.at/\~{}gerald/}}

\dedicatory{To Vladimir Aleksandrovich Marchenko and Leonid Andreevich Pastur,
our teachers and inspiring colleagues.}
\thanks{Work supported by the Austrian Science Fund (FWF) under Grants No.\ Y330 and J2655.}


\begin{abstract}
We develop direct and inverse scattering theory for Jacobi operators
with steplike coefficients which are asymptotically close to
different finite-gap quasi-periodic coefficients on different sides.
We give a complete characterization of the scattering data, which
allow unique solvability of the inverse scattering problem in the
class of perturbations with finite first moment.
\end{abstract}

\maketitle

\section{Introduction}

In this paper we consider direct and inverse scattering theory for Jacobi operators with
steplike quasi-periodic finite-gap background, using the Marchenko \cite{mar} approach.

Scattering theory for Jacobi operators is a classical topic with a long tradition. Originally developed
on an informal level by Case in \cite{dinv4}, the first rigorous results for the case of a
constant background were given by Guseinov \cite{gu} with further extensions by
Teschl \cite{tivp}, \cite{tjac}. The case of periodic backgrounds was completely
solved in \cite{voyu} (who in fact handle almost periodic operators with a homogenous
Cantor type spectrum) respectively \cite{emtqps} using different approaches. Moreover,
the case of a steplike situation, where the coefficients are asymptotically close to
two different quasi-periodic finite-gap operators, was solved in \cite{emtstep} (see
also \cite{baeg}, \cite{eg}) under the restriction that the two background operators are isospectral.
It is the purpose of the present paper to remove this restriction.

We should also mention that scattering theory for Jacobi operators is directly
applicable to the investigation of the Toda lattice with initial data in the above mentioned classes. See
for example \cite{bdmek}, \cite{dkkz}, \cite{vdo} for steplike constant backgrounds,
and \cite{emtist}, \cite{emtsr}, \cite{kateptr}, \cite{katept}, and \cite{mtqptr} for periodic backgrounds.
For further possible applications and additional references we refer to the discussion in \cite{emtstep}.

Finally, let us give a brief overview of the remaining sections.
After recalling some necessary facts on algebro-geometric quasi-periodic finite-gap
operators in Section~\ref{secQP}, we construct the transformation operators and
investigate the properties of the scattering data in Section~\ref{secSD}. In Section~\ref{secGLM}
we derive the Gel'fand-Levitan-Marchenko equation and show that it uniquely determines
the operator. In addition, we formulate necessary conditions for the scattering data to uniquely
determine our Jacobi operator. Our final Section~\ref{secINV} shows that our necessary
conditions for the scattering data are also sufficient.

\section{Step-like finite-band backgrounds}
\label{secQP}

First we need to recall some facts on quasi-periodic finite-band Jacobi operators
which contain all periodic operators as a special case. We refer to
\cite[Chapter~9]{tjac} and \cite{emtqps} for details.

Let $H_q^\pm$ be two quasi-periodic finite-band Jacobi operators,\footnote{Everywhere in this
paper the sub or super index "$+$" (resp.\ "$-$") refers to the background on the right
(resp.\ left) half-axis.}
\be
H_q^\pm f(n) = a_q^\pm(n) f(n+1) + a_q^\pm(n-1) f(n-1) + b_q^\pm(n) f(n),
\quad f\in\ell^2(\Z),
\ee
associated with the Riemann surface of the square root
\begin{equation}\label{defP}
P_\pm(z)= -\prod_{j=0}^{2g_\pm+1} \sqrt{z-E_j^\pm}, \qquad
E_0^\pm < E_1^\pm < \cdots < E_{2g_\pm+1},
\end{equation}
where $g_\pm\in \N$ and $\sqrt{.}$ is the standard root with branch
cut along $(-\infty,0)$. In fact, $H_q^\pm$ are uniquely determined by
fixing a Dirichlet divisor
$\sum_{j=1}^{g^\pm}(\mu_j^\pm, \si_j^\pm)$, where
$\mu_j^\pm\in[E_{2j-1}^\pm,E_{2j}^\pm]$ and $\si_j^\pm\in\{-1, 1\}$.
The spectra of $H_q^\pm$ consist of $g_\pm+1$ bands
\be \label{0.1}
\si_\pm:=\sig(H_q^\pm) = \bigcup_{j=0}^{g_\pm}
[E_{2j}^\pm,E_{2j+1}^\pm].
\ee
We will identify the set
$\C\setminus\sig(H_q^\pm)$ with the upper sheet of the Riemann surface.
The upper and lower sides of the cuts over the spectrum are denoted by
$\siu$ and $\sil$ and the symmetric points on these cuts by
$\lau$ and $\lal$, that is,
\begin{equation}\nn
f(\lam^u)=\lim_{\epsilon \downarrow 0}f(\lam + i\epsilon), \quad
f(\lam^l)=\lim_{\epsilon \downarrow 0}f(\lam - i\epsilon), \quad
\lam \in \sig_\pm.
\end{equation}

We will develop the scattering theory for the operator
\be \label{0.2}
H f(n)=a(n-1) f(n-1) + b(n) f(n)+ a(n) f(n+1), \quad n\in\Z,
\ee
whose coefficients are asymptotically close to the
coefficients of $H_q^\pm$ on the corresponding half-axes:
\be \label{hypo}
\sum_{n = 0}^{\pm \infty} |n| \Big(|a(n) -
a_q^\pm(n)| + |b(n) - b_q^\pm(n)| \Big) < \infty.
\ee
The special case $H_q^- = H_q^+$ has been exhaustively studied in \cite{emtqps}
(see also \cite{voyu}) and the case where $H_q^-$ and $H_q^+$ are in the
same isospectral class $\si_-=\si_+$ was treated in \cite{emtstep}.
Several results are straightforward generalizations, in such
situations we will simply refer to \cite{emtqps}, \cite{emtstep} and
only point out possible differences.

Let $\psi_q^\pm(z,n)$ be the Floquet solutions of the spectral equations
\be\label{0.12}
H_q^\pm \psi(n)=z \psi(n), \qquad z\in\C,
\ee
that decay for $z\in\C\setminus\si_\pm$ as $n\to \pm\infty$.
They are uniquely defined by the condition $\psi_q^\pm(z,0)=1$,
$\psi_q^\pm(z,\cdot)\in\ell^2(\Z_\pm)$. The solution $\psi_q^+(z,n)$
(resp.\ $\psi_q^-(z,n)$)  coincides with the upper (resp.\ lower)
branch of the Baker--Akhiezer functions of $H_q^+$ (resp.\ $H_q^-$),
see \cite{tjac}. The second solutions $\breve \psi_q^\pm(z,n)$ are
given by the other branch of the Baker--Akhiezer
functions and satisfy $\breve \psi_q^\pm(z,\cdot)\in\ell^2(\Z_\mp)$ as
$z\in\C\setminus\si_\pm$. Their Wronskian is equal to
\beq\label{2.9}
W_q^\pm(\breve{\psi}_q^\pm(z), \psi_q^\pm(z)) =
\pm\frac{1}{\rho_\pm(z)},
\eeq
where
\begin{equation}\label{0.5}
\rho_\pm(z)=\frac{\prod_{j=1}^{g_\pm}(z-\mu_j^\pm)}{ P_\pm(z)}
\end{equation}
satisfy by our choice of the branch for the square root
\be \label{0.10}
\im(\rho_\pm(\lau))>0, \quad \im(\rho_\pm(\lal))<0,
\quad  \la\in\si_\pm.
\ee
In \eqref{2.9} the following notation is
used
\beq \label{new11}
W_{q,n}^\pm(f,g):=a_q^\pm(n) \left(f(n)
g(n+1) - f(n+1)g(n)\right).
\eeq

Note that $\psi_q^\pm(z,n)$,
$\breve\psi_q^\pm(z,n)$ have continuous limits as
$z\to\laul\in\siul_\pm\setminus\partial\si_\pm$, where
$$\partial\si_\pm=\{E_0^\pm,...,E_{2g_\pm+1}^\pm\},$$
and they satisfy the symmetry property
\be \label{0.6}
\psi_q^\pm(\lal,n)= \ol{\psi_q^\pm(\lau,n)}=\breve\psi_q^\pm(\lau,n), \quad \lam\in\si_\pm.
\ee

The points $(\mu_j^\pm,\si_j^\pm)$, $1\le j\le g_\pm$, form the divisors of poles of the
Baker--Akhiezer functions. Correspondingly, the sets of Dirichlet
eigenvalues $\{\mu_1^\pm,...,\mu_{g_\pm}^\pm\}$ can
be divided in three disjoint subsets
\begin{align}
\begin{split}
M^\pm &= \{\mu_j^\pm \,|\, \mu_j^\pm \in \R\backslash\sig_\pm \mbox{
is a pole of } \psi_q^\pm(z,1)\},\\
\breve{M}^\pm &= \{ \mu_j^\pm\,|\,\mu_j^\pm
 \in \R\backslash\sig_\pm \mbox{ is a pole of } \overline{\psi_q^\pm(z,1)}\},\\
\hat{M}^\pm &= \{ \mu_j^\pm \,|\, \mu_j^\pm \in \partial \si_\pm \}.
\end{split}
\end{align}
In order to remove the singularities of $\psi_q^\pm(z,n)$, $\breve\psi_q^\pm(z,n)$
we introduce
\begin{align}\label{0.63}
\begin{split}
\delta_\pm(z) &:= \prod_{\mu^\pm_j\in M_\pm}  (z - \mu^\pm_j),\\
\hat \delta_\pm(z) &:= \prod_{\mu^\pm_j\in M_\pm}
(z-\mu_j^\pm) \prod_{\mu^\pm_j \in \hat M_\pm} \sqrt{z -
\mu^\pm_j},\\
\breve \delta_\pm(z) &:= \prod_{\mu^\pm_j\in
\breve M_\pm} (z-\mu_j^\pm) \prod_{\mu^\pm_j \in \hat M_\pm} \sqrt{z
- \mu^\pm_j},
\end{split}
\end{align}
where $\prod =1$ if there are no multipliers, and set
\be
\tilde\psi_q^\pm(z,n)=\delta_\pm(z) \psi_q^\pm(z,n),\quad
\hat\psi_q^\pm(z,n)=\hat\delta_\pm(z) \psi_q^\pm(z,n).
\ee

\begin{lemma}\label{lem1.1}
The Floquet solutions $\psi_q^\pm$, $\breve\psi_q^\pm$ have the following properties:
\begin{enumerate}
\item
The functions $\psi_q^\pm(z,n)$ (resp.\ $\breve\psi_q^\pm(z,n)$) are
holomorphic as functions of $z$ in the domain $\mathbb{C}\setminus
(\si_\pm\cup M_\pm)$ (resp.\ $\mathbb{C}\setminus (\si_\pm\cup\breve
M_\pm)$), take real values on the set $\mathbb{R}\setminus
\sigma_\pm$, and have simple poles at the points of the set $M_\pm$
(resp.\ $\breve M_\pm$). They are continuous up to the boundary
$\sipmu\cup \sipml$ except at the points in $\hat M_\pm$ and
satisfy the symmetry property \eqref{0.6}. For $E \in \hat M_\pm$,
they satisfy
\[
\psi_q^\pm(z,n)=O\left(\frac{1}{\sqrt{z-E}}\right), \quad
\breve\psi_q^\pm(z,n) =O\left(\frac{1}{\sqrt{z-E}}\right), \quad
z\to E\in \hat M_\pm.
\]
Moreover, the estimate
\begin{equation}\label{new7}
\hat\psi_q^\pm(z,n) - \hat\psi_q^\pm(E,n) = O(\sqrt{z-E}), \qquad E\in\pa\si_\pm,
\end{equation}
is valid.
\item
The following asymptotic expansions hold as $z \rightarrow \pm
\infty$
\begin{equation}  \label{1.12}
\psi_q^\pm(z,n) = z^{\mp n}\Big(\vprod{j=0}{n-1} a_q^\pm(j)
\Big)^{\pm 1} \Big(1 \pm \frac1z  \vsum{j=0}{n-1} b_q^\pm(j +
{\scriptstyle{1 \atop 0}}) + O(\frac{1}{z^2})\Big),
\end{equation}
where
\begin{equation*}
 \vprod{j=n_0}{n-1}f(j) = \left\{ \begin{array}{c@{\ }l}
\prod\limits_{j=n_0}^{n-1} f(j), & n > n_0, \\ 1, & n=n_0,\\
\prod\limits_{j=n}^{n_0-1} f(j)^{-1}, & n < n_0, \end{array} \right.
 \quad
\vsum{j=n_0}{n-1} f(j) = \left\{ \begin{array}{c@{\ }l}  \sum\limits_{j=n_0}^{n-1}
f(j), & n > n_0, \\ 0, & n = n_0, \\ -\sum\limits_{j=n}^{n_0-1} f(j), & n < n_0.\end{array}
\right.
\end{equation*}
\item
The functions $\psi_q^\pm(\la,n)$ form a complete orthogonal system
on the spectrum with respect to the weight
\be \label{1.121}
d\omega_\pm(\la) = \frac{1}{2\pi\I} \rho_\pm(\la) d\la,
\ee
namely
\begin{equation}
\oint_{\sig_\pm}\overline{\psi_q^\pm(\lam,m)}\psi_q^\pm(\lam,n)d\omega_\pm(\la)
= \delta(n,m), \label{1.14}
\end{equation}
where
\be \label{1.141}
\oint_{\si_\pm}f(\la) d\la := \int_{\sipmu} f(\lau) d\la - \int_{\sipml} f(\lal) d\la.
\ee
Here $\delta(n,m)=1$ if $n=m$ and $\delta(n,m)=0$ else is the Kronecker
delta.
\end{enumerate}
\end{lemma}

\section{Scattering data}
\label{secSD}

Now let $H$ be a steplike operator with coefficients
$a(n)$, $b(n)$ satisfying \eqref{hypo}. The two solutions $\psi_\pm(z,n)$ of the spectral
equation
\be \label{0.11}
H \psi = z \psi, \quad z\in\C,
\ee
which are asymptotically close to the Floquet solutions $\psi_q^\pm(z,n)$ of the background
equations \eqref{0.12} as $n\to\pm\infty$, are called Jost solutions. They can
be represented as (see \cite{emtqps})
\be \label{3.16}
\psi_\pm(z,n) = \sum_{m=n}^{\pm \infty} K_\pm(n,m) \psi_q^\pm(z, m),
\ee
where the
functions $K_\pm(n, .)$ are real valued and satisfy the estimate
 \be \label{estimate K q}
|K_\pm(n,m)| \leq C_\pm(n) \sum_{j=[\frac{m+n}{2}] }^{\pm \infty}
\Big(|a(j)-a_q^\pm(j)| + |b(j)-b_q^\pm(j)|\Big), \quad \pm m
> \pm n>0.
\ee
The functions $C_\pm(n) > 0$ decrease monotonically as
$n\rightarrow \pm \infty$.
Moreover, we have
\begin{align} \label{0.13}
\begin{split}
a(n) &= a_q^+(n)\frac{K_+(n+1, n+1)}{K_+(n,n)}, \\
a(n) &= a_q^-(n)\frac{K_-(n, n)}{K_-(n+1,n+1)}, \\
b(n) &= b_q^+(n) + a_q^+(n)\frac{K_+(n, n+1)}{K_+(n,n)} -
a_q^+(n-1) \frac{K_+(n-1, n)}{K_+(n-1,n-1)}, \\
b(n) &= b_q^-(n) + a_q^-(n-1)\frac{K_-(n, n-1)}{K_-(n,n)} - a_q^-(n)
\frac{K_-(n+1, n)}{K_-(n+1,n+1)},
\end{split}
\end{align}
which implies (cf.\ \cite{emtqps}) the following asymptotic behavior
of the Jost solutions as $z\to\pm \infty$ using \eqref{3.16}, \eqref{1.12},
\be                 \label{B4jost}
\psi_\pm(z,n) =   z^{\mp n}K_\pm(n,n)
\Big(\vprod{j=0}{n-1}a_q^\pm(j)\Big)^{\pm 1} \Big(1 + \Big(B_\pm(n)
\pm \vsum{j=1}{n} b_q^\pm(j- {\scriptstyle{0 \atop 1}})
\Big)\frac{1}{z} + O(\frac{1}{z^2}) \Big), \ee where
\beq
\label{0.14}
B_\pm(n)= \sum_{m=n\pm 1}^{\pm\infty} (b_q^\pm(m)-b(m)).\eeq
For  $\la \in \siu_\pm\cup\sil_\pm$ a second pair of solutions of \eqref{0.11} is given
by
\be
\breve\psi_\pm(\la,n)=\sum_{m=n}^{\pm \infty} K_\pm(n,m) \breve\psi_q^\pm(\la, m),
\quad \la \in \siu_\pm\cup\sil_\pm,
\ee
which cannot be continued to the complex plane.
Note that $\breve\psi_\pm(\la,n)=\overline{\psi_\pm(\la,n)}$, $\la \in \sig_\pm$,
and from \eqref{hypo}, \eqref{3.16} we conclude
\be\label{0.62}
W(\overline{\psi_\pm(\la)}, \psi_\pm(\la)) = W_q^\pm(\breve{\psi}_q^\pm(\la), \psi_q^\pm(\la)) =
\pm \rho_\pm(\la)^{-1}.
\ee

The Jost solutions $\psi_\pm$ are holomorphic in the domains
$\mathbb{C}\setminus \left(\si_\pm\cup M_\pm\right)$ and inherit
almost all properties of their background counterparts listed in
Lemma~\ref{lem1.1}. As before, we set
\be\label{0.64}
\tilde\psi_\pm(z,n)=\delta_\pm(z)
\psi_\pm(z,n),\quad \hat\psi_\pm(z,n)=\hat\delta_\pm(z) \psi_\pm(z,n).
\ee

The following Lemma is proven in \cite{emtqps}.

\begin{lemma}\label{lem2.1}
The Jost solutions  have the following properties.
\begin{enumerate}
\item
For all $n$, the functions $\psi_\pm(z,n)$ are holomorphic in the
domain $\mathbb{C}\setminus (\si_\pm\cup M_\pm)$ with respect to $z$
and continuous up to the boundary $(\sipmu\cup
\sipml)\setminus\pa\si_\pm$, where \be \label{symm}
\psi_\pm(\lau,n)=\overline{\psi_\pm(\lal,n)}, \quad \la \in
(\sipmu\cup \sipml)\setminus\pa\si_\pm. \ee The functions
$\psi_\pm(z,n)$ are real valued for $z \in \mathbb{R}\setminus
\sigma_\pm$ and have simple poles at $\mu_j \in M_\pm$. Moreover,
$\hat\psi_\pm$ are continuous up to the boundary $\sipmu\cup
\sipml$.
\item
At the band edges we have for $ \la\in\siul_\pm$
\beq\label{new3}\begin{array}{lll}
&\psi_\pm(\la,n) -\overline{\psi_\pm(\la,n)}=o(1), &
E\in\pa\si_\pm\setminus\hat M_\pm, \\
&\psi_\pm(\la,n) +\overline{\psi_\pm(\la,n)}=o\Big(\frac{1}{\sqrt{\la - E}}\Big),&
E\in\hat M_\pm.\end{array}\eeq
\end{enumerate}
\end{lemma}

Next, we introduce the sets
\be \label{2.5}
\si^{(2)}:=\si_+\cap\si_-, \quad
\si_\pm^{(1)}=\clos\,(\si_\pm\setminus\si^{(2)}), \quad
\si:=\si_+\cup\si_-,
\ee
where $\si$ is the (absolutely) continuous
spectrum of $H$ and $\si_+^{(1)}\cup \si_-^{(1)}$ resp.\
$\si^{(2)}$ are the parts which are of multiplicity one resp.\ two.
We will denote the interior of the spectrum by $\inte(\si)$, that is,
$\inte(\si):=\si\setminus\pa\si$.

In addition to the continuous part, $H$ has a finite number of
eigenvalues situated in the gaps, $\si_d=\{\la_1,...,\la_p\} \subset \R \setminus \si$
(see, e.g., \cite{tosc}).
For every eigenvalue we introduce the corresponding norming constants
\be \label{norming}
\ga_{\pm, k}^{-1}=\sum_{n \in
\Z}|\tilde\psi_\pm(\la_k, n)|^2, \quad 1 \leq k \leq p.
\ee
Moreover, $\tilde{\psi}_{\pm}(\la_k, n) = c_k^{\pm} \tilde{\psi}_{\mp}(\la_k, n)$
 with $c_k^+ c_k^- = 1$.

Let
\beq\label{2.8}
W(z):=W(\psi_-(z), \psi_+(z))
\eeq
be the Wronskian of two Jost solutions. This function is meromorphic in
the domain $\mathbb{C}\setminus\si$ with possible poles at the
points $M_+\cup M_-\cup(\hat M_+\cap\hat M_-)$
and with possible square root singularities at the points $\hat
M_+\cup\hat M_-\setminus (\hat M_+\cap \hat M_-)$. Set
\beq \label{2.10}
\tilde W(z) =W(\tilde\psi_-(z), \tilde\psi_+(z)), \quad \hat W(z) =
W(\hat\psi_-(z), \hat\psi_+(z)),
\eeq
then $\hat W(\la)$ is holomorphic in
the domain $\mathbb{C}\setminus\mathbb{R}$
and continuous up to the boundary. But unlike to $W(z)$
and $\tilde W(z)$, the function  $\hat W(\la)$ may not take real
values on the set $\mathbb{R}\setminus\si$ and complex conjugated
values on the different sides of the spectrum. That is why it is
more convenient to characterize the spectral properties of the
operator $H$ by means of the function $\tilde W$, which can have
singularities at the points of the sets $\hat M_+\cup\hat M_-$.
We will study the precise character of these singularities in
Lemma~\ref{lem2.3} below.

Note that outside the spectrum the function $\tilde W(z)$ vanishes precisely at the
eigenvalues. However, it might also vanish inside the spectrum at points in $\pa\si_-\cup\pa\si_+$.
We will call such points virtual levels of the operator $H$,
\beq\label{2.363}
\si_v:= \{ E\in\si:\ \hat W(E)=0\},
\eeq
and we will show that $\si_v \subseteq \pa\si\cup(\pa\si_+^{(1)}\cap\pa\si_-^{(1)})$
in Lemma~\ref{lem2.3}.
All other points $E$ of the set $\pa\si_+\cup\pa\si_-$ correspond to
the generic case $\hat W(E)\neq 0$.

Our next aim is to derive the properties of the scattering matrix. Introduce the
scattering relations
\be\label{6.16}
T_\mp(\la) \psi_\pm(\la,n)
=\overline{\psi_\mp(\la,n)} + R_\mp(\la)\psi_\mp(\la,x),
\quad\la\in\simpul,
\ee
where the transmission and reflection coefficients are defined as usual,
\be\label{6.17}
T_\pm(\la):= \frac{W(\overline{\psi_\pm(\la)}, \psi_\pm(\la))}{W(\psi_\mp(\la), \psi_\pm(\la))},\quad
R_\pm(\la):= - \frac{W(\psi_\mp(\la),\overline{\psi_\pm(\la)})}{W(\psi_\mp(\la),\,\psi_\pm(\la))}, \quad\la\in \sipmul.
\ee
The equalities in \eqref{6.17} imply the identity
$$
\frac{1}{T_+(\la) \rho_+(\la)} = \frac{1}{T_-(\la) \rho_-(\la)} =W(\la), \quad \la\in\si^{(2)},
$$
where $W(\la)$ is the Wronskian of two Jost solutions \eqref{2.8}.
This Wronskian plays an important role in the characterization
of the properties of the scattering matrix. Namely, the following result is valid.

\begin{lemma}\label{lem2.3}
The entries of the scattering matrix have the following properties:

{\bf I}.
$$
\begin{array}{lcl}
{\bf (a)}&T_\pm(\lau) =\overline{T_\pm(\lal)},& \la\in\si_\pm,  \\
& R_\pm(\lau) =\overline{R_\pm(\lal)},&\la\in\si_\pm, \nn\\
{\bf (b)}&\displaystyle{\frac{T_\pm(\la)}{\overline{T_\pm(\la)}}} = R_\pm(\la),&
\la\in\si_\pm^{(1)},\nn\\
{\bf (c)}& 1 - |R_\pm(\la)|^2 =
\displaystyle{\frac{\rho_\pm(\la)}{\rho_\mp(\la)}}\,|T_\pm(\la)|^2,&\la\in\si^{(2)},\\
{\bf (d)}&\overline{R_\pm(\la)}T_\pm(\la) +
R_\mp(\la)\overline{T_\pm(\la)}=0,& \la\in\si^{(2)}.
\end{array}
$$

{\bf II}.
The functions $T_\pm(\la)$ can be extended analytically to the domain
$\mathbb{C} \setminus (\si\cup M_\pm\cup\breve M_\pm)$ and satisfy
\be \label{2.18}
\frac{1}{T_+(z) \rho_+(z)} = \frac{1}{T_-(z) \rho_-(z)} =W(z).
\ee
The function  $W(z)$  has the following properties:

{\bf (a)} The function $\tilde W(z) = \delta_+(z)\delta_-(z) W(z)$
 is holomorphic on $\mathbb{C}\setminus\si$ with simple zeros at the eigenvalues $\la_k$, where
\be \label{2.11}
\bigg(\frac{ d\tilde W}{d z}(\la_k)\bigg)^2
=\frac{1}{\ga_{+,k}\ga_{-,k}}.
\ee
Moreover,
\be \label{6.9}
\overline{\tilde W(\lau)}=\tilde W(\lal), \quad \la\in\si,
\qquad \tilde W(z)\in\mathbb{R}, \quad
z\in\mathbb{R}\setminus \sigma.
\ee

{\bf (b)} The function $\hat W(z) = \hat\delta_+(z) \hat\delta_-(z)
W(z)$ is continuous on  the set $\mathbb{C}\setminus\si$ up to the
boundary $\siu\cup\sil$. It can have zeros on the set
$\pa\si\cup(\pa\si_+^{(1)}\cap\pa\si_-^{(1)})$ and does not vanish
at the other points of the spectrum $\si$.  If $\hat W(E)=0$ as
$E\in\pa\si\cup(\pa\si_+^{(1)}\cap\pa\si_-^{(1)})$, then \be
\label{estim} \frac{1}{\hat W(\la)} = O\left(\frac{1}{\sqrt{\la -
E}}\right),\quad\mbox{where} \quad \la \in \si\quad \mbox{close
to}\quad E.\ee

{\bf (c)} In addition,
\be \label{T infty}
T_+(\infty)=T_-(\infty)>0.
\ee

{\bf III}.
{\bf (a)} The reflection coefficients $R_\pm(\la)$  are continuous functions on
$\inte(\sipmul)$.

{\bf (b)} If $E\in\pa\si_+\cap \partial\si_-$ and $\hat W(E)\neq 0$, then the functions $R_\pm(\la)$
are also continuous at $E$. Moreover,
\be\label{P.1}
R_\pm(E) = \left\{ \begin{array}{c@{\quad\mbox{for}\quad}l}
-1 & E\notin\hat M_\pm,\\
1 & E\in\hat M_\pm. \end{array}\right.
\ee
\end{lemma}

\begin{proof}
{\bf I}.
The symmetry property {\bf (a)} follows from formulas (\ref{6.17}) and (\ref{symm}).
For {\bf (b)}, use \eqref{6.17} and observe that $\psi_\mp(\la)$ are real valued for
$\la\in\inte(\si_\pm^{(1)})$.
Let $\la\in\inte(\si^{(2)})$.
By (\ref{6.16}),
$$
|T_\pm|^2 W(\psi_\mp,\overline{\psi_\mp}) = (|R_\pm|^2 -1)
W(\psi_\pm,\overline{\psi_\pm}),
$$
and property {\bf (c)} follows from (\ref{0.62}). The
consistency condition {\bf (d)} can be derived directly from
definition (\ref{6.17}).

{\bf II}.
The identity (\ref{2.18}) follows from (\ref{6.17}).
{\bf (a)} The Wronskian inherits the properties of $\psi_\pm(z)$,
so it remains to show (\ref{2.11}).
If $\hat W(z_0)=0$ for $z_0\in\mathbb{C}\setminus\si$, then
\beq\label{D.1}
\tilde\psi_\pm(z_0,n)=c^\pm\tilde\psi_\mp(z_0,n)
\eeq
for some constants $c^\pm$ (depending on $z_0$), which satisfy $c^-c^+ =1$.
In particular, each zero of $\tilde W$ (or $\hat W$) outside the continuous spectrum
is a point of the discrete spectrum of $H$ and vice versa.

Let $\ga_{\pm,j}$ be the norming constants defined in (\ref{norming})
for some point of the discrete spectrum
$\la_j$. By virtue of \cite{tjac}, Lemma~2.4,
\begin{align}\nn  \label{der W q}
\frac{d}{dz}W(\tilde {\psi}_-(z), \tilde{\psi}_+(z)) \Big|_{\la_j} &=
W_n(\tilde {\psi}_-(\la_j), \tfrac{d}{dz}\tilde{\psi}_+(\la_j))
+W_n(\tfrac{d}{dz} \tilde {\psi}_-(\la_j), \tilde{\psi}_+(\la_j))\\
&= - \sum_{k \in \Z} \tilde {\psi}_-(\la_j, k)
\tilde{\psi}_+(\la_j, k)
= - \frac{1}{c_j^{\pm}\gamma_{\pm, j}}.
\end{align}
Since $c_j^-c_j^+=1$, we obtain (\ref{2.11}).

 {\bf (b)} Continuity of $\hat W$ up to the boundary follows from the corresponding
property of $\hat \psi_\pm(z,n)$. We begin with the investigation of
the possible zeros of this function on the spectrum.

First let $\la_0\in\inte(\si^{(2)}):= \si^{(2)}\setminus\pa\si^{(2)}$,
 that is, $\hat\delta_-\neq 0$ and $\hat\delta_+\neq 0$. Suppose
 $W(\la_0) = 0$, then $\psi_+(\la_0,n)=c\,\psi_-(\la_0,n)$ and
$\overline{\psi_+(\la_0,n)}=\bar c\, \overline{\psi_-(\la_0,n)}$,
i.e. $W(\psi_+,\overline {\psi_+})=|c|^2 W(\psi_-,\overline
{\psi_-})$. But this implies opposite signs for $\rho_+, \rho_-$
by (\ref{0.62}), $\sign \rho_+(\la_0)= -\sign$ $\rho_-(\la_0)$,
which contradicts (\ref{0.10}).

Let $\la_0\in\inte(\si^{(1)}_\pm)$ and $\tilde W(\la_0) =0$. The
point $\la_0$ can coincide with a pole $\mu\in M_\mp$. But
$\psi_\pm(\la_0,n)$ and $\overline{\psi_\pm(\la_0,n)}$ are linearly
independent and bounded, and $\tilde\psi_\mp(\la_0,n)\in\mathbb{R}$.
If $W(\la_0)=0$, then
$\tilde\psi_\mp=c_1^{\pm}\,\psi_\pm=c_2^\pm\,\overline{\psi_\pm}$
which implies $W(\psi_\pm,\overline{\psi_\pm})(\la_0)=0$, a contradiction.

In the general mutual location of the background spectra the case
$\la_0=E\in (\pa\si^{(2)}\cap\inte(\si_\pm))\subset\inte(\si)$ can occur. If
$\hat W(E)=0$, then $W( \psi_\pm, \hat\psi_{\mp})(E)=0$, where
$\hat\psi_{\mp}$ are defined by (\ref{0.64}). The values of
$\hat\psi_{\mp}(E,\cdot)$ are either purely real or purely imaginary,
therefore $W( \overline{\psi_\pm}, \hat\psi_{\mp})(E)=0$, that is,
$\overline{\psi_\pm(E,n)}$ and $\psi_\pm(E,n)$ are linearly
dependent, which is impossible at inner points of the set $\si_\pm$.

Thus, the virtual level $\si_v$ of $H$ defined in \eqref{2.363} can only
be located on the set $\pa\si_-\cap\pa\si_+$, that is,
\be\label{in}
\si_v \subseteq \pa\si\cup\big(\pa\si_-^{(1)}\cap\pa\si_+^{(1)}\big).
\ee

To prove \eqref{estim}, take $E\in\si_v$ and assume for example $E\in\si_+$.
By \eqref{6.16} and \eqref{2.18},
\[
\frac{\hat\delta_+(\la)\hat\psi_-(\la,n)}{\hat\delta_-(\la)\rho_+(\la)W(\la)}=
\hat\delta_+(\la) \overline{\psi_+(\la,n)}
+R_+(\la)\hat\psi_+(\la,n).
\]
Choose $n_0$ such that $\hat\psi_-(E,n_0)\neq 0$.
By continuity we also have $\hat\psi_-(\la,n_0)\neq 0$ in a small vicinity of $E$.
Then
\[
\frac{\hat\delta_+(\la) \overline{\psi_+(\la,n_0)}
+R_+(\la)\hat\psi_+(\la,n_0)}{\hat\psi_-(\la,n_0)}=O(1),\quad\la\to
E.
\]
Correspondingly,
\[
\frac{1}{\hat W(\la)}=O\bigg(\frac{\prod_{j=1}^{g_+}(\la -
\mu_j^+)}{\hat\delta^2_+(\la)\sqrt{\la-E}}\bigg)=O\left(\frac{1}{\sqrt{\la
- E}}\right), \quad \la\in\si_+,
\]
which proves \ref{estim}.

{\bf (c)}  Equation (\ref{T infty}) follows from (\ref{B4jost}).

{\bf III}. {\bf (a)} follows from the corresponding properties of
$\psi_\pm(z)$ and from {\bf II}, {\bf (b)}. To show {\bf III}, {\bf (b)} we
use that by (\ref{6.17}) the reflection coefficients have the
representation
\beq\label{2.n}
R_\pm(\la)=-\frac{W(\ov{\psi_\pm(\la)},\psi_\mp(\la))}
{W\left(\psi_\pm(\la),\psi_\mp(\la)\right)}=\mp\frac{W(\ov{\psi_\pm(\la)},
\psi_\mp(\la))}{W(\la)}
\eeq and are continuous on both sides of the
set $\inte(\si_\pm)\setminus (M_\mp\cup\hat M_\mp)$.
Moreover,
$$
|R_\pm(\la)|=\bigg|\frac{W(\ov{\hat\psi_\pm(\la)},\hat\psi_\mp(\la))}{\hat W(\la)}\bigg|,
$$
where the denominator does not vanish on the set $\si_\pm\setminus\si_v$. Hence
$R_\pm(\la)$ are continuous on this set since both numerator and denominator are.

Next, let $E\in\pa\si_\pm \setminus \si_v$ (in particular $\hat
W(E)\ne 0$). Then, if $E\notin \hat M_\pm$, we use (\ref{2.n}) in
the form
\beq \label{2.n1}
R_\pm(\la)= -1 \mp \frac{\hat{\de}_\pm(\la)
W(\psi_\pm(\la)-\ov{\psi_\pm(\la)},\hat\psi_\mp(\la))} {\hat W(\la)},
\eeq
which shows $R_\pm(\la) \to -1$ since
$\psi_\pm(\la)-\ov{\psi_\pm(\la)} \to 0$ by Lemma~\ref{lem2.1}, (2).
This settles the first case in (\ref{P.1}). Similarly, if
$E\in\hat M_\pm$, we use (\ref{2.n}) in the form
\beq \label{2.n2}
R_\pm(\la)= 1 \pm \frac{\hat{\de}_\pm(\la)
W(\psi_\pm(\la)+\ov{\psi_\pm(\la)},\hat\psi_\mp(\la))} {\hat
W(\la)},
\eeq
which shows $R_\pm(\la) \to -1$ since
$\hat{\de}_\pm(\la)=O(\sqrt{\la - E})$ and
$\psi_\pm(\la)+\ov{\psi_\pm(\la)} = o\big(\frac{1}{\sqrt{\la - E}}\big)$
by Lemma~\ref{lem2.1}, (2). This settles the second case in (\ref{P.1}) as well.
 \end{proof}

\section{The Gel'fand-Levitan-Marchenko equation}
\label{secGLM}

The aim of this section is to derive the inverse scattering problem
equation (the Gel'fand-Levitan-Marchenko equation) and to discuss
some additional properties of the scattering data which are
consequences of this equation.

\begin{theorem}
The inverse scattering problem (the GLM) equation has the form
\be \label{glm1 q}
K_\pm(n,m) + \sum_{l=n}^{\pm \infty}K_\pm(n,l)F_\pm(l,m) =
\frac{\delta(n,m)}{K_\pm(n,n)}, \qquad \pm m \geq \pm n,
\ee
where
\begin{align} \nn
F_\pm(m,n) &= \oint_{\sig_\pm} R_\pm(\lam) \psi_q^\pm(\lam,m)
\psi_q^\pm(\lam,n) d\omega_\pm\\ \label{glm2 q} & +
\int_{\sig_\mp^{(1),u}} |T_\mp(\lam)|^2 \psi_q^\pm(\lam,m)
\psi_q^\pm(\lam,n) d\omega_\mp + \sum_{k=1}^p \ga_{\pm,k}
\tilde\psi_q^\pm(\la_k,n) \tilde\psi_q^\pm(\la_k,m).
\end{align}
\end{theorem}

\begin{proof}
Consider a closed contour $\Gamma_\epsilon$ consisting of a large circular
arc and some contours inside this arc, which envelope the spectrum
$\sigma$ at a small distance $\varepsilon$ from the spectrum.
Let $\pm m \geq\pm n$. The residue theorem, \eqref{1.121},
\eqref{B4jost}, \eqref{2.11}, and equality
$\tilde \psi_\mp(\la_k,n)=c_j^\mp \tilde \psi_\pm(\la_k,n)$ imply
\begin{align} \nn
\frac{1}{2\pi\I} \oint_{\Gamma_\epsilon}
\frac{\psi_\mp(\lam,n)\psi_q^\pm(\lam,m)} {W(\la)} d\lam &=
\frac{\delta(n,m)}{K_\pm(n,n)} + \sum_{k=1}^p
\res_{\la_k}\bigg(\frac{\tilde \psi_\mp(\lam,n) \tilde \psi_q^\pm(\lam,m)}
{\tilde W(\la)} \bigg)\\ \label{4.3} &= \frac{\delta(n,m)}{K_\pm(n,n)} -
\sum_{k=1}^p \ga_{\pm,k} \tilde\psi_\pm(\la_k,n)
\tilde\psi_q^\pm(\la_k,m),
\end{align}
since the integrand is meromorphic on $\C\backslash \si$ with simple
poles at the eigenvalues $\la_k$ and at $\infty$ if $m=n$. It is
continuous till the boundary except at the points
$E\in \partial \sigma_+ \cup \partial \sigma_-$ where
\begin{equation}
\frac{\psi_\mp(\lam,n)\psi_q^\pm(\lam,m)} {W(\la)} = O
\left(\frac{1}{\sqrt{\lam - E}}\right), \quad E \in \partial
\sigma_+ \cup \partial \sigma_-,
\end{equation}
by \eqref{estim}. On the other hand, as $\epsilon \rightarrow 0$,
\begin{align} \label{4.5}
\begin{split}
& \frac{1}{2\pi\I} \oint_{\si}
\frac{\psi_\mp(\lam,n)\psi_q^\pm(\lam,m)}
{W(\lam)} d\lam = \\
&\quad = \frac{1}{2\pi\I} \oint_{\sig_\pm} \frac{\big(\ol{\psi_\pm(\lam,n)}
+ R_\pm(\lam)\psi_\pm(\lam,n)\big)\psi_q^\pm(\lam,m)}
{T_\pm(\lam) W(\lam)} d\lam \\
&\qquad + \frac{1}{2\pi\I} \oint_{\sig_\mp^{(1)}} \frac{\psi_\mp(\lam,n)\psi_q^\pm(\lam,m)}
{W(\lam)} d\lam  \\
&\quad = \oint_{\sig_\pm} \ol{\psi_\pm(\lam,n)}\psi_q^\pm(\lam,m)d\omega^\pm
+ \oint_{\sig_\pm}
R_\pm(\lam)\psi_\pm(\lam,n)\psi_q^\pm(\lam,m)d\omega^\pm\\
&\qquad + \frac{1}{2\pi\I}\int_{\sig_\mp^{(1),u}}\psi_q^\pm(\lam,m)
\bigg(\frac{\psi_\mp(\lam,n)}{W(\lam)} -
\frac{\ol{\psi_\mp(\lam,n)}}{\ol{W(\lam)}}
\bigg)d\lam.
\end{split}
\end{align}
It remains to treat the last integrand. By \eqref{6.16} and Lemma~\ref{lem2.3},~{\bf I},
\begin{equation} \nn
\ol{\psi_\mp(\lam,n)} = T_\mp(\lam)\psi_\pm(\lam,n) - R_\mp(\lam)\psi_\mp(\lam,n)
=  T_\mp(\lam)\psi_\pm(\lam,n) - \frac{T_\mp(\lam)}{\ol{T_\mp(\lam)}}\psi_\mp(\lam,n),
\end{equation}
and therefore
\begin{equation} \nn
\frac{\psi_\mp(n)}{W} -
\frac{\ol{\psi_\mp(n)}}{\ol{W}}
= \frac{\ol{W}\ol{T_\mp} + W T_\mp}{|W|^2\ol{T_\mp}}\psi_\mp(n) -
\frac{T_\mp}{\ol{W}}\psi_\pm(n) = - \frac{T_\mp}{\ol{W}}\psi_\pm(n),
\end{equation}
since $\ol{W}\ol{T_\mp} + W T_\mp = 2 \re(W T_\mp)=0$ on $\sig_\mp$. In summary, \eqref{4.3}
and \eqref{4.5} yield
\begin{align} \nn
\frac{\delta(n,m)}{K_\pm(n,n)} &= K_\pm(n,m)
+ \oint_{\sig_\pm} R_\pm(\lam)\psi_\pm(\lam,n)\psi_q^\pm(\lam,m)d\omega^\pm \\ \nn
&+ \int_{\sig_\mp^{(1),u}} |T_\mp(\lam)|^2\psi_\pm(\lam,n)\psi_q^\pm(\lam,m)d\omega^\mp
+ \sum_{j=1}^p \ga_{\pm,j} \tilde \psi_\pm(\la_j,n) \tilde \psi_q^\pm(\la_j,m),
\end{align}
and applying \eqref{3.16} finishes the proof.
\end{proof}

As it is shown in \cite{emtqps}, the estimate \eqref{estimate K q} for $K_\pm(n,m)$
implies the following estimates for $F_\pm(n,m)$.

\begin{lemma}\label{lem4.1}
The kernel of the GLM equation satisfies the following properties.\\
{\bf IV}. There exist  functions $C_\pm(n)>0$ and $ q_\pm(n)\geq 0$,
$n\in\Z_\pm$, such that $C_\pm(n)$ decay as $n\to\pm\infty$,
$|n| q(n) \in \ell^1(\Z_\pm)$, and
\begin{align} \label{fourest}
\begin{split}
&|F_{\pm}(n, m)|\leq C_\pm(n)\sum_{j=n+m}^{\pm \infty} q(j), \\
&\sum_{n = n_0}^{\pm \infty}|n| \big| F_{\pm}(n,n) -
 F_{\pm}(n \pm 1, n \pm 1)\big| < \infty,    \\
&\sum_{n = n_0}^{\pm \infty}|n| \big|a_q^\pm(n)  F_{\pm}(n,n+1) -
a_q^\pm(n-1)  F_{\pm}(n - 1, n)\big| < \infty.
\end{split}
\end{align}
\end{lemma}

In summary, we have obtained the following necessary conditions for
the scattering data:

\begin{theorem}\label{theor1}
The scattering data
\begin{align}\label{4.6}
\begin{split}
{\mathcal S} &= \Big\{ R_+(\la),\,T_+(\la),\,
\la\in\si_+^{\mathrm{u,l}}; \,
R_-(\la),\,T_-(\la),\, \la\in\si_-^{\mathrm{u,l}};\\
& \qquad \la_1,\dots, \la_p \in\mathbb{R} \setminus (\si_+\cup\si_-),\,
\ga_{\pm, 1}, \dots, \ga_{\pm, p} \in\mathbb{R}_+\Big\}
\end{split}
\end{align}
satisfy the properties {\bf I-III} listed in Lemma~\ref{lem2.3}. The
functions $F_\pm(n,m)$, defined in (\ref{glm2 q}), satisfy property
{\bf IV} in Lemma~\ref{lem4.1}.
\end{theorem}

In fact, the conditions on the scattering data given in Theorem~\ref{theor1}
are both necessary and sufficient for the solution of the
scattering problem in the class (\ref{hypo}). The sufficiency of
these conditions as well as the algorithm for the solution of the
inverse problem will be discussed in the next section.

\section{The inverse scattering problem}
\label{secINV}

Let $H_q^\pm$ be two arbitrary quasi-periodic Jacobi operators
associated with sequences  $a_q^\pm(n), b_q^\pm(n)$ as introduced in
Section~\ref{secQP}. Let $\mathcal{S}$ be given scattering data with
corresponding kernels $F_\pm(n,m)$ satisfying the necessary
conditions of Theorem~\ref{theor1}.

First we show that the GLM equations (\ref{glm1 q}) can be solved for
$K_\pm(n,m)$ if $F_\pm(n,m)$ are given.
\begin{lemma}\label{lem5.4}
Under condition {\bf IV}, the GLM equations (\ref{glm1 q}) have
unique real-valued solutions $K_\pm(n,\cdot)\in \ell^1(n,\pm\infty)$
satisfying the estimates
\beq \label{5.100}
|K_\pm(n,m)|\leq
C_\pm(n)\sum_{j=\left[\frac{n+m}{2}\right]}^{\pm\infty} q(j),
\quad\pm m>\pm n.
\eeq
Here the functions $q_\pm(n)$ and $C_\pm(n)$ are of the same
type as in (\ref{fourest}).

Moreover, the following estimates are valid
\begin{align} \label{invest}
\begin{split}
&\sum_{n = n_0}^{\pm \infty}|n| \big| K_{\pm}(n,n) -
K_{\pm}(n \pm 1, n \pm 1)\big| < \infty,    \\
&\sum_{n = n_0}^{\pm \infty}|n| \big|a_q^\pm(n)  K_{\pm}(n,n+1) -
a_q^\pm(n-1)  K_{\pm}(n - 1, n)\big| < \infty.
\end{split}
\end{align}
\end{lemma}

\begin{proof}
The solvability of (\ref{glm1 q}) under condition (\ref{fourest})
and the estimates (\ref{5.100}), \eqref{invest} follow
completely analogous to the corresponding
result in \cite[Theorem 7.5]{emtqps}. To prove
uniqueness, first note that the GLM equations are generated by
compact operators. Thus, it is sufficient to prove that the
equation
\beq \label{5.102}
f(m)+\sum_{\ell=n}^{\pm\infty} F_\pm(\ell,m)
f(\ell) =0
\eeq
has only the trivial solution in the space
$\ell^1(n,\pm\infty)$. The proof is similar for the "$+$" and "$-$"
cases, hence we give it only for the "$+$" case. Let $f(\ell)$, $\ell > n$,
be a nontrivial solution of (\ref{5.102}) and set $f(\ell)=0$ for
$\ell \leq n$. Since $F_+(\ell,n)$ is real-valued, we can assume that
$f(\ell)$ is real-valued. Abbreviate by
\beq \label{5.103}
\widehat  f(\la)
=\sum_{m \in \mathbb{Z}} \psi_q^+(\la,m) f(m)
\eeq
the generalized Fourier transform, generated by the spectral
decomposition (\ref{1.14}) (cf.\ \cite{T}).
Recall that $\widehat f(\la)\in L^1_{loc} (\siu_+\cup\sil_+)$.

Multiplying (\ref{5.102}) by $f(m)$, summing over $m \in \Z$,
and applying \eqref{1.14}, (\ref{glm2 q}), (\ref{5.103}), and
condition {\bf I}, {\bf (a)}, we have
\begin{align}
\begin{split}
& 2 \int_{\si_+^u}|\widehat f(\la)|^2 d\omega_+(\la) + 2 \re
\int_{\si_+^u}R_+(\la) \widehat f(\la)^2 d\omega_+(\la) \\ \label{5.104}
& \quad +\int_{\si_-^{(1),u}}\widehat f(\la)^2 |T_-(\la)|^2 d\omega_-(\la)
+ \sum_{k=1}^p \ga_{+,k} \bigg(\sum_{n \in \mathbb{Z}} \tilde \psi_q^+(\la_k,n) f(n) \bigg)^2=0.
\end{split}
\end{align}
The last two summands in (\ref{5.104}) are nonnegative since
$\widehat  f(\la) \in\mathbb{R}$ for $\la\in\si_-^{(1)}$
and $\tilde \psi_q^+(\la_k)\in \R$.
We estimate the first two integrands by
\[
|\widehat f(\la)|^2 + \re R_+(\la) \widehat f(\la)^2 \geq |\widehat f(\la)|^2 - |R_+(\la)
 \widehat f(\la)^2|
\geq \big(1-|R_+(\la)|\big)|\widehat f(\la)|^2
\]
and drop the last summand in (\ref{5.104}), thus obtaining
\beq \label{5.105}
2\int_{\si^{(2),\mathrm{u}}} (1-|R_+(\la)|)|\widehat f(\la)|^2 d\omega_+(\la)
+ \int_{\si_-^{(1),\mathrm{u}}} \widehat f(\la)^2 |T_-(\la)|^2 d\omega_-(\la) \leq 0.
\eeq
Here we also used that
\[
\int_{\si_+^{(1),\mathrm{u}}}(1-|R_+(\la)|) |\widehat f(\la)|^2
d\omega_+(\la) = 0,
\]
which follows from condition {\bf I}, {\bf (b)}. Since $|R_+(\la)|<1$
for $\la\in \inte(\si^{(2)})$ and
$\omega_-(\la)>0$ for $\la\in\inte(\si_-^{(1)})$ we conclude that
\[
\widehat f(\la)=0 \quad \mbox{for}\quad \la\in\si^{(2)}\cup
\si_-^{(1)}= \si_-.
\]
The function $\widehat f(z)$ can be defined by formula (\ref{5.103})
as a meromorphic function on $\mathbb{C}\setminus\si_+$. By our
analysis it is even meromorphic on $\mathbb{C}\setminus\si^{(1)}_+$
and vanishes on $\si_-$. Thus $\widehat f(z)$ and hence also $f(m)$
are equal to zero.
\end{proof}

Next, define the sequences $a_\pm, b_\pm$ by
\begin{align} \label{5.1}
\begin{split}
a_+(n) &= a_q^+(n)\frac{K_+(n+1, n+1)}{K_+(n,n)}, \\
a_-(n) &= a_q^-(n)\frac{K_-(n, n)}{K_-(n+1,n+1)}, \\
b_+(n) &= b_q^+(n) + a_q^+(n)\frac{K_+(n, n+1)}{K_+(n,n)}
- a_q^+(n-1) \frac{K_+(n-1, n)}{K_+(n-1,n-1)}, \\
b_-(n) &= b_q^-(n) + a_q^-(n-1)\frac{K_-(n, n-1)}{K_-(n,n)} - a_q^-(n)
\frac{K_-(n+1, n)}{K_-(n+1,n+1)},
\end{split}
\end{align}
and note that estimate (\ref{invest}) implies
\beq  \label{5.2}
n\left\{|a_\pm(n)-a_q^\pm(n)| + |b_\pm -
b_q^\pm(n)|\right\}\in\ell^1(\Z_\pm).
\eeq

\begin{lemma}\label{lem5.7}
The functions $\psi_\pm(z,n)$, defined by
\beq  \label{Dop}
\psi_\pm(z,n) =\sum_{m=n}^{\pm\infty}
K_\pm(n,m) \psi_q^\pm(z,m),
\eeq
solve the equations
\beq \label{5.3}
a_\pm(n-1)\psi_\pm(z,n-1)+b_\pm(n)\psi_\pm(z,n)+a_\pm(n)\psi_\pm(z,n+1)=z\psi_\pm(z,n),
\eeq
where $a_\pm(n), b_\pm(n)$ are defined by (\ref{5.1}).
\end{lemma}

\begin{proof}
Consider the two operators\footnote{We don't know that $ H_\pm$ is limit
point at $\mp\infty$ yet, but this will not be used.}
$$
 (H_\pm y)(n)= a_\pm(n-1)y_\pm(n-1)+b_\pm(n)y_\pm(n)+a_\pm(n)y_\pm(n+1),\quad
n\in\mathbb{Z}.
$$
Define two discrete integral operators
$$
\left ({\mathcal K}_\pm f\right)(n) =\sum_{m=n}^{\pm\infty} K_\pm(n,m) f(m).
$$
Then (cf.\ \cite{emtqps})  the following identity is valid
$$
 H_\pm{\mathcal K}_\pm={\mathcal K}_\pm\,H_q^\pm,
$$
which implies (\ref{5.3}).
\end{proof}

The remaining problem is to show that $a_+(n)\equiv a_-(n)$,
$b_+(n)\equiv b_-(n)$ under conditions {\bf II} and {\bf III} on the
scattering data $\mathcal{S}$.

\begin{theorem}\label{theor2}
Let the scattering data ${\mathcal S}$, defined as in (\ref{4.6}),
satisfy conditions {\bf I}, {\bf (a)--(c)}, {\bf II}, {\bf III}, {\bf (a)},
and {\bf IV}. Then each of the GLM equations (\ref{glm1 q}) has
unique solutions $K_\pm(n,m)$, satisfying the estimate
(\ref{invest}). The functions $a_\pm(n), b_\pm(n)$, defined by
(\ref{5.1}), satisfy (\ref{5.2}).

Under the additional conditions  {\bf III}, {\bf (b)} and {\bf I}, {\bf (d)},
these functions coincide, $a_+(n)\equiv a_-(n)=a(n)$, $b_+(n)\equiv
b_-(n)=b(n)$, and the data ${\mathcal S}$ are the scattering data
for the Jacobi operator associated with the sequences $a(n), b(n)$.
\end{theorem}

The proof of Theorem~\ref{theor2} takes up the remaining
section and is split into several lemmas for the convenience of the
reader.

To prove uniqueness of the reconstructed potential we follow the
method proposed in \cite{mar}. Recall that, by Lemma~\ref{lem1.1}
(iii), the functions $\psi_q^\pm(\lam,n)$ form an orthonormal basis
with corresponding generalized Fourier transform. Split the kernel
of the GLM equation (\ref{glm2 q}) into three summands $F_\pm(m,n)=
F_{r,\pm}(m,n) + F_{h,\pm}(m,n) +F_{d,\pm}(m,n)$ and set \be
G_\pm(n, m) := \sum_{l=n}^{\pm \infty}K_\pm(n,l) F_{r,\pm}(l,n). \ee
Then one obtains as in \cite[Theorem 8.2]{emtqps} that the functions
$h_\mp(\la, n)$, defined by
\begin{align} \label{5.11}
\begin{split}
h_\mp(\lam, n) &=  \frac{1}{T_\pm(\lam)}\Bigg(
\frac{\breve{\psi}_q^\pm(\lam,n)}{K_\pm(n, n)} + \sum_{m = n\mp1}^{\mp
\infty}G_\pm(n, m) \breve{\psi}_q^\pm(\lam,m) \\
&\quad \mp \int_{\sig_\mp^{(1),u}} |T_\mp(\xi)|^2 \psi_\pm(\xi,n)
\frac{W_{q,n-1}^\pm(\psi_q^\pm(\xi), \breve{\psi}_q^\pm(\lam))} {\xi-\lam}
d\omega_\mp(\xi)\\
&\quad \pm \sum_{k=1}^p \gamma_{\pm,
k} \tilde \psi_\pm(\lam_k, n) \frac{W_{q,n-1}^\pm(\tilde \psi_q^\pm(\lam_k),
\breve{\psi}_q^\pm(\lam))} {\lam - \lam_k }\Bigg),
\end{split}
\end{align}
satisfy
\be \label{5.10}
T_\pm(\lam)h_\mp(\lam, n) = \ol{\psi_\pm(\lam, n)} + R_\pm(\lam) \psi_\pm(\lam, n),
\quad \lam \in \sig_\pm^{u,l}.
\ee

Despite the fact that $h_\mp(\la,n)$ are defined via the background solutions
corresponding to the opposite half-axis $\mathbb{Z}_\pm$, they share a series
of properties with $\psi_\mp(\la,n)$.
Namely, we prove

\begin{lemma}\label{lem5.1}
Let $h_\mp(z,n)$ be defined by formula (\ref{5.11}) on the set $\sipmul$.
\begin{enumerate}
\item
The functions $\tilde h_\mp(z,n)= \de_\mp(z) h_\mp(z,n)$
admit analytic extensions to the domain $\mathbb{C}\setminus\si$.
\item
The functions $\tilde h_\mp(z,n)$
are continuous up to the boundary $\siul$ except possibly at the points $\pa\si_+\cup\pa\si_-$.
Furthermore,
\begin{align} \label{5.30}
\begin{split}
&\tilde h_\mp(\lau,n)=\tilde h_\mp(\lal,n)\in \R, \qquad
\la\in\mathbb{R}\setminus\si_\mp,\\
& \tilde h_\mp(\lau,n)=\overline{\tilde h_\mp(\lal,n)}, \qquad \la\in\inte(\si_\mp).
\end{split}
\end{align}
\item
For large $z$ the functions $h_\mp(z,n)$ have the following asymptotic
behavior
\be \label{h infty}
h_{\mp}(z, n) = \frac{z^{\pm n}}{K_\pm(n,n)T_\pm(\infty)}
\Big(\vprod{j=0}{n-1}a_q^\pm(j)\Big)^{\mp1}\Big(1+O(\frac{1}{z})\Big),
\quad z \rightarrow \infty.
\ee
\item
We have
\begin{align*}
W^\pm(h_\mp(z),\psi_\pm(z)):=&a_\pm(n)\big(h_\mp(z,n)\psi_\pm(z,n+1)-
h_\mp(z,n+1)\psi_\pm(z,n)\big)\\
\equiv& \pm W(z),
\end{align*}
where $W(z)$ is defined by (\ref{2.18}).
\end{enumerate}
\end{lemma}

\begin{remark}\label{rem5}
Note that we did not establish the connection between the function
$W(z)$ and the functions $W^\pm(\psi_+(z,n),\psi_-(z,n))$,
which can depend on $n$, because $\psi_+$ and $\psi_-$ are
the solutions of Jacobi equations corresponding to possibly
different operators $H_+$ and $H_-$.
\end{remark}

\begin{proof}
(i). To show that $\tilde h_\mp(z,n)$ have analytic extensions to
$\mathbb{C}\setminus\si$, we study each term in (\ref{5.11}) separately.

First of all, note that due to the representation
\beq \label{2.19}
T_\pm(z) = \frac{1}{\rho_\pm(z) W(z)} =
\frac{\hat\de_\mp(z)}{\breve \de_\pm(z)}
\frac{\sqrt{\prod_{j=0}^{2g_\pm+1} (z-E_j^\pm)}}{\hat W(z)},
\eeq
the functions $\tilde\zeta_\mp(z,n)= \de_\mp(z) \zeta_\mp(z,n)$, where
\beq \label{5.141}
\zeta_\mp(z,n):=\frac{\breve{\psi}_q^\pm(z,n)}{T_\pm(z)},
\eeq
can be continued analytically to $\mathbb{C}\setminus\si$. This also
holds for the second term since $G_\pm(n,\cdot)\in \ell^1(\mathbb{Z})$ are
real-valued.

Next we discuss the properties of the Cauchy-type integral in the
representation (\ref{5.11}). We represent the third summand in (\ref{5.11})
multiplied by $T_\pm^{-1}(z)$ as
\beq \label{5.123}
\Theta_\mp(z,n):=\mp \frac{1}{2\pi\I} \int_{\si_\mp^{(1),\mathrm{u}}}
\theta_\mp(z,\xi,n)\frac{d\xi}{\xi - z},
\eeq
where
\begin{align}\nn
\theta_\mp(z,\xi,n) &= -\frac{\de_\mp(\xi)^2}{\rho_\mp(\xi) |\tilde
W(\xi)|^2} \tilde\psi_\pm(\xi,n)
W_{q,n-1}^\pm(\tilde\psi_q^\pm(\xi,\cdot), \zeta_\mp(z,\cdot))\\
\label{5.13} & = -\frac{|\hat\de_\mp(\xi)|^2}{\rho_\mp(\xi) |\hat
W(\xi)|^2} \frac{|\hat\de_\pm(\xi)|^2}{\hat\de_\pm(\xi)^2}
\hat\psi_\pm(\xi,n) W_{q,n-1}^\pm(\hat\psi_q^\pm(\xi,\cdot),
\zeta_\mp(z,\cdot)).
\end{align}
By property {\bf II}, {\bf (a)} the function $\hat W(\xi)$ has no zeros in the interior of
$\si_\mp^{(1),\mathrm{u}}$. Thus, for $z\notin \si_\mp^{(1)}$, the functions $\theta_\mp(z,.,n)$
are bounded in the interior of $\si_\mp^{(1)}$ and the only possible singularities can arise
at  the boundary. We claim
\beq\label{5.131}
\theta_\mp(z,\xi,n)= \begin{cases}
O(\sqrt{\xi - E}) & \mbox{ for } E\notin\si_v,\\
O\big(\frac{1}{\sqrt{\xi - E}}\big) & \mbox{ for } E\in\si_v,
\end{cases}
\qquad E\in\pa\si_\mp^{(1)},\; z\neq E.
\eeq
This follows from $\frac{|\hat\delta_\mp(\xi)|^2}{\rho_\mp(\xi)}=O(\sqrt{\xi - E})$
together with $\hat W(\xi) = O(1)$ if $E\notin\si_v$ and
$1/\hat W(\xi) = O(1/\sqrt{\xi - E})$ by {\bf II}, {\bf (b)} if $E\in\si_v$.
Therefore, $\theta_\mp$ are integrable and the third summand of (\ref{5.11})
also inherits the properties of $\zeta_\mp(z,n)$.

Finally, the last summand in (\ref{5.11}) again inherits the properties of
$\tilde\zeta_\mp(z,n)$ except for possible additional poles at the eigenvalues $\la_k$.
However, these cancel with the zeros of $\tilde W(z)$ at $z=\lam_k$.

(ii). We consider the boundary values next. The only nontrivial term is of course
the Cauchy-type integral (\ref{5.123}) as $z\to\la\in\inte(\si_\mp^{(1)})$.
First of all observe that by \eqref{2.9} and \eqref{2.18},
\[
\frac{W_{q,n-1}^\pm(\tilde\psi_q^\pm(\la),\breve\psi_q^\pm(z))}{T_\pm(z)} \to (\de_\pm W)(\la),
\]
where the functions $\de_\pm W$ are bounded and nonzero for $\la\in\inte(\si_\mp^{(1)})$ by
{\bf II}, {\bf (a)}. Hence the Plemelj formula applied to (\ref{5.123}) gives
\[
\Theta_\mp(\la,n) = \pm \frac{\tilde\psi_\pm(\la,n)}{2\de_\pm(\la)
\rho_\mp(\la) \overline{W(\la)}} \mp
\dashint_{\si_\mp^{(1),
\mathrm{u}}}\frac{\theta_\mp(\la,\xi,n)}{\xi - \la}d\xi, \quad
\la\in\inte(\si_\mp^{(1),\mathrm{u}}),
\]
where both terms are finite. Here $\dashint$ denotes the principle value integral.
Therefore, the boundary values away from $\pa\si_+\cup\pa\si_-$ exist and we have
\beq \label{5.14}
h_\mp(\lau,n)= \overline{h_\mp(\lal,n)}, \quad\la\in\si_+\cup\si_-.
\eeq
By property {\bf I}, {\bf (b)},
\beq \label{5.16}
h_\mp=T_\pm^{-1}\left(R_\pm\psi_\pm +\overline{\psi_\pm}\right) =
\frac{\psi_\pm}{\ov{T_\pm}} +\frac{\overline{\psi_\pm}}{T_\pm}\in\mathbb{R},
\quad  \la\in\si_\pm^{(1)},
\eeq
from which
\beq \label{5.15}
h_\mp(\lau,n)= h_\mp(\lal,n) ,\quad \la\in\si_\pm^{(1)},
\eeq
follows. Combining (\ref{5.14}) and (\ref{5.15}) yields (\ref{5.30}).

(iii). Since the last two terms in (\ref{5.11}) are $O(z^{-1})$, the asymptotic behavior
follows from \eqref{B4jost} and {\bf II}, {\bf (c)}.

(iv). From (\ref{5.10}), \eqref{0.62}, and (\ref{2.18}) we obtain
\beq \nonumber
W^\pm(h_\mp(\la), \psi_\pm(\la))=\frac{W^\pm(\ol{{\psi}_\pm(\la)}, \psi_\pm(\la))}{T_\pm(\la)} =
\frac{1}{T_\pm(\la)\rho_\pm(\la)} = \pm W(\la), \quad \la\in \si_\pm.
\eeq
Hence equality holds for all $z\in\mathbb{C}$ by analytic continuation.
\end{proof}

\begin{corollary}\label{col7}
The functions $\tilde h_\mp(z,n)$ admit analytic extensions to
$\mathbb{C}\setminus\si_\mp$.
\end{corollary}

\begin{proof}
Property (i) of Lemma \ref{lem5.1} holds for
$z\in\mathbb{C}\setminus \si$. Relation (\ref{5.30}) implies that
$\tilde h_\mp$ have no jumps across $z\in\inte(\si_\pm^{(1)})$. To
finish the proof we need to show that the possible remaining
singularities at $E\in\pa\si_\pm^{(1)}\cap\pa\si$ are removable.
This follows from (cf.\ \eqref{2.19})
\beq \label{5.142}
\hat\zeta_\mp(z,n)= \frac{\hat W(z)}{\sqrt{\prod_{j=0}^{2g_\pm+1}
(z-E_j^\pm)}} \breve\de_\pm(z) \breve\psi_q^\pm(z,n)
\eeq
which shows
$\tilde\zeta_\mp(z,n)=O((z-E)^{-1/2})$ and hence
$\tilde h_\mp(z,n)= O((z-E)^{-1/2})$ for $E\in\si_\pm^{(1)}\cap\pa\si$.

However, let us emphasize at this point that the behavior of
$h_\pm(z,n)$ at the remaining edges is a more subtle question
to be discussed later.
\end{proof}

Eliminating $\overline{\psi_\pm}$ from
$$
\left\{\begin{array}{lll}
\overline{ R_\pm(\la)}\,\overline{\psi_\pm(\la,n)} + \psi_\pm(\la,n)&=&\overline{h_\mp(\la,n)}\,
\overline{T_\pm(\la)}\\[2mm]
R_\pm(\la)\,\psi_\pm(\la,n) + \overline{\psi_\pm(\la,n)} &=&h_\mp(\la,n)\,T_\pm(\la)
\end{array}\right.
$$
yields
$$
\psi_\pm(\la,n)\left(1 - |R_\pm(\la)|^2\right) =\overline{h_\mp(\la,n)}\,
\overline{T_\pm(\la)} -
\overline{R_\pm(\la)}\,h_\mp(\la,n)\,T_\pm(\la).
$$
We apply {\bf I}, {\bf (c)}, {\bf II}, and the consistency condition
{\bf I}, {\bf (d)} to obtain
\begin{align} \nonumber
T_\mp(\la)\psi_\pm(\la,n) &=\overline{h_\mp(\la,n)} - \frac{\overline{R_\pm(\la)}
T_\pm(\la)}{\ov{T_\pm(\la)}} h_\mp(\la,n)\\ \label{4.26}
&=\overline{h_\mp(\la,n)} + R_\mp(\la) h_\mp(\la,n),
\quad\la\in\si^{(2)}.
\end{align}
This equation together with (\ref{5.10}) gives us a system from
which we can eliminate the reflection coefficients $R_\pm$. We obtain
\beq\label{5.20}
T_\pm(\la) \big(\psi_\pm(\la)\psi_\mp(\la) -
h_\pm(\la)h_\mp(\la)\big)=
\psi_\pm(\la)\overline{h_\pm(\la)} -
\overline{\psi_\pm(\la)}h_\pm(\la), \quad
\la\in\si^{(2),\mathrm{u,l}}.
\eeq
Now introduce the function
\beq \label{5.211}
G(z):= G(z,n) = \frac{\psi_+(z,n)\psi_-(z,n) - h_+(z,n)h_-(z,n)}{W(z)}
\eeq
which is well defined in the domain $z\in \mathbb{C}\setminus
\left(\si\cup\si_d\cup M_+\cup M_- \right)$.
By (\ref{5.20}) and \eqref{2.18},
\beq \label{5.22}
G(\la) = \left(\psi_\pm(\la)\overline{h_\pm(\la)} -
\overline{\psi_\pm(\la)}h_\pm(\la)\right) \rho_\pm(\la),
\quad\la\in\si^{(2),\mathrm{u,l}},
\eeq
so we need to study the properties of $G(z,n)$ as a function of $z$.
Our aim  is to prove that $G(z,n)=0$, which will follow
from the next lemma.

\begin{lemma}\label{lem5.2}
The function $G(z,n)$, defined by \eqref{5.211}, has the following properties.
\begin{enumerate}
\item
$G(\lau,n) =  G(\lal,n)\in\mathbb{R}$ for $\la\in\mathbb{R}\setminus (\pa\si_-\cup\pa\si_+\cup\si_d)$.

\item
It has removable singularities at the points $\pa\si_-\cup\pa\si_+\cup\si_d$,
where $\si_d:=\{\la_1,...,\la_p\}$.
\end{enumerate}
\end{lemma}

\begin{proof}
(i). We can rewrite $G(z,n)$ as
\beq\label{5.25}
G(z,n)= \frac{\tilde\psi_+(z,n)\tilde\psi_-(z,n) -
\tilde h_+(z,n)\tilde h_-(z,n)}{\tilde W(z)},
\eeq
where $\tilde h_\pm(z,n)=\de_\pm(z)h_\pm(z,n)$ as usual.
The numerator is bounded near the points under consideration and the
denominator does not vanish there.
Thus $G(z,n)$ has no singularities at the points $(M_+\cup M_-) \setminus \si_d$.

Furthermore, by Lemma~\ref{lem5.1}, {\bf II}, {\bf (a)}, and Lemma~\ref{lem2.1}
we know that $G(z,n)$ has continuous limiting values on the sets
$\si_-$ and $\si_+$, except possibly at the edges, and satisfies
$$
G(\lau,n) =\overline{G(\lal,n)},\quad \la\in\si_+ \cup \si_-.
$$
Hence, if we can show that these limits are real, they will be equal and
$G(z,n)$ will extend to a meromorphic function on $\mathbb{C}$, that is, (i) holds.
To this aim we first observe that \eqref{5.30}, \eqref{5.22}, and Lemma~\ref{lem2.1}
imply
\beq\label{5.24}
G(\lau,n) = G(\lal,n)\in\mathbb{R}, \quad \la\in\inte(\si^{(2)}).
\eeq
Thus, it remains to prove
\beq\label{5.26}
G(\lau,n) =  G(\lal,n)\in\mathbb{R}
\quad \mbox{for}\quad \la\in\inte(\si_-^{(1)})\cup\inte(\si_+^{(1)}).
\eeq
Let us show that $G(\la,n)$ has no jump on the set
$\inte(\si_-^{(1)})\cup\inte(\si_+^{(1)})$. We abbreviate
\beq\label{5.28}
\left[G\right] :=G(\la)
-\overline{G(\la)}=\left[\frac{\psi_+\psi_-}{W}\right] -
\left[\frac{h_+h_-}{W}\right], \quad \la\in\si_\pm^{(1),\mathrm{u}},
\eeq
and drop some dependencies until the end of this lemma for notational simplicity.

Let $\la\in\inte(\si_\mp^{(1),\mathrm{u}})$, then $\psi_\pm,h_\pm\in\mathbb{R}$
and $\overline T_\mp=-(\overline W\,\rho_\mp)^{-1}$.
By (\ref{2.18}), {\bf (I)}, {\bf (b)}, and
(\ref{5.10}) we obtain  for $\la\in\inte(\si_\mp^{(1)})$
\beq \label{5.29}
\left[\frac{\psi_+\psi_-}{W}\right]=  \psi_\pm \left[\frac{\psi_\mp}{W}\right] = \rho_\mp\psi_\pm
\left(\psi_\mp T_\mp+ \overline\psi_\mp \overline T_\mp\right)=
\rho_\mp h_\pm \psi_\pm |T_\mp|^2.
\eeq
Since $\rho_\pm\in\mathbb{R}$ for $\la\in\inte(\si_\mp^{(1),\mathrm{u}})$, (\ref{2.18}) implies
$$
\left[\frac{h_\mp}{W}\right] = \rho_\pm \left[h_\mp T_\pm\right].
$$
The only non-real summand in (\ref{5.11}) is the
Cauchy-type integral. The Plemelj formula applied to this
integral gives
$$
\left[h_\mp T_\pm\right]= -
\rho_\mp\psi_\pm |T_\mp|^2 W(\psi_q^\pm,\breve\psi_q^\pm)=
\rho_\mp \psi_\pm|T_\mp|^2 \frac{1}{\rho_\pm},
$$
and by (\ref{5.29}) we get
\beq\label{5.301}
\left[\frac{h_+h_-}{W}\right]=\left[\frac{\psi_+\psi_-}{W}\right]
= \rho_\mp \psi_\pm h_\pm |T_\mp|^2,\quad \la\in\inte(\si_\mp^{(1)}).
\eeq
Since $\tilde W\neq 0$ for $\la\in\inte(\si_\mp^{(1)})$, the function
$$
\rho_\mp \psi_\pm h_\pm |T_\mp|^2= - \frac{\delta_\mp^2}{\rho_\mp}
\frac{\tilde\psi_\pm \tilde h_\pm}{|\tilde W|^2}
$$
is bounded on the set under consideration. Finally,
(\ref{5.301}) and (\ref{5.28}) imply (\ref{5.26}).

(ii). Now we prove that the function $G(z,n)$ has removable
singularities at the points $\pa\si_-\cup\pa\si_+\cup\si_d$. We divide
this set into four subsets
\beq \label{5.32}
\Omega_1^\pm=\pa\si^{(2)}\cap\inte(\si_\mp), \quad
 \Omega_2=\pa\si^{(2)}\cap\pa\si, \quad \Omega_3^\pm=\pa\si_\pm^{(1)}\cap\pa\si_\pm, \quad
\Omega_4=\si_d. \eeq

Since all singularities of $G$ are at most isolated poles, it is
sufficient to show that
\beq \label{5.35}
G(z)=o\big((z - E)^{-1}\big)
\eeq
from some direction in the complex plane.

$\Omega_1$: Consider $E \in \Omega_1^+$ (the case $E \in \Omega_1^-$ being
completely analogous). We will study $\lim_{\la \rightarrow E}G(\la,n)$ as $\la \in \inte(\si^{(2)})$
using \eqref{5.22} with the ``$-$'' sign. Note that $\psi_-=O(1)$, $\rho_-=O(1)$, and $\hat W(E)\neq 0$.
Moreover, we obtain from Lemma~\ref{lem2.1} respectively {\bf II} that
$$
\psi_+(\la) = \left\{ \begin{array}{cl}  O(1), & E \notin \hat M_+,\\
O\left(\frac{1}{\sqrt{\la - E}}\right), & E \in \hat M_+, \end{array} \right.
\quad
\frac{1}{T_+(\la)} = \left\{ \begin{array}{cl} O\left(\frac{1}{\sqrt{\la - E}}\right),
 & E \notin \hat M_+,\\
O(1), & E \in \hat M_+, \end{array} \right.
$$
which shows
$$
h_-(\la)=\frac{\ov{\psi_+(\la)}+R_+(\la)\psi_+(\la)}{T_+(\la)}=
O\left(\frac{1}{\sqrt{\la - E}}\right)
$$
for $\la \in \si^{(2)}$. Inserting this into \eqref{5.22} shows
 $G(\la,n)=O\big(\frac{1}{\sqrt{\la - E}}\big)$
and finishes the case $E\in \Omega_1$.

$\Omega_2$: For $E \in \pa\si^{(2)}\cap\pa\si$, we use \eqref{5.22}
and take the limit $\la \rightarrow E$ from $\si^{(2)}$. First of all,
observe that
$$
\breve\de_-\left(R_-\psi_- +\ol\psi_-\right)=\left\{\begin{array}{ll}
O(1) & E\in \si_v,\\
o(1) & E\notin\si_v.\end{array}\right.
$$
The case $E\in\si_v$ is evident. If $E\notin \si_v$ then \eqref{new3} and \eqref{P.1}
yield
$$\breve\de_-\left(R_-\psi_- +\ol\psi_-\right)=
\left\{\begin{array}{cl} \breve\de_-\big((\psi_- - \ol\psi_-) +(R_- + 1)\psi_-\big)
, & \quad  E\notin \hat M_-\\
\big(\breve\de_-(\psi_- + \ol\psi_-) +(R_- -
1)\breve\de_-\psi_-\big), & \quad E\in\hat
M_-\end{array}\right. =o(1).
$$
Therefore, both for virtual and non-virtual levels the estimate
\beq\label{new4}
\breve\de_-\left(R_-\psi_- +\ol\psi_-\right)\hat W=o(1),\quad E\in\pa\si_-,
\eeq
is valid. Inserting \eqref{5.10} into the summand $\ol\psi_+h_+\rho_+$ of
\eqref{5.22} (for the second summand we use an analogous approach) we obtain
(recall \eqref{defP})
\begin{align} \nn
\ol\psi_+h_+\rho_+ &=\ol\psi_+\rho_+ \rho_- (\ol \psi_- + R_-
\psi_-) W = \frac{\ol\psi_+\breve\de_+}{P_+P_-} \hat \de_+
\hat \de_- \breve\de_- (\ol \psi_- + R_- \psi_-) W\\ \label{new2}
&= \frac{\ol\psi_+\breve\de_+}{P_+P_-}
\breve\de_-\left(R_-\psi_- +\ol\psi_-\right)\hat W.
\end{align}
Combining the estimate
$$
\frac{\ol\psi_+\breve\de_+}{P_+P_-}=O\bigg(\frac{1}{\la -E}\bigg)
$$
with \eqref{new4} we have $G(z)=o\big((z-E)^{-1}\big)$ as desired.

$\Omega_3$:  Suppose that $E \in \pa \si_-^{(1)}\cap\pa\si_-$ (the case $E
\in \pa \si_+^{(1)}\cap\pa\si_+ $ is again analogous). Now we
cannot use \eqref{5.22}, so we proceed directly from formula
\eqref{5.211} estimating the summands $\frac{\psi_+\psi_-}{W}$ and
$\frac{h_+h_-}{W}$ separately. We investigate the limit as
$\la\to E$ from the set $\inte(\si_-^{(1)})$.
By Lemma~\ref{lem2.1} and \eqref{estim} we have
\beq \label{5.53}
\frac{\psi_+\psi_-}{W}=\frac{\hat \psi_+\hat \psi_-}{\hat W} =
O\left(\frac{1}{\sqrt{\la - E}}\right),
\eeq
hence the first summand has the desired behavior. To estimate the second summand,
we split the function $h_-(\la,n)$ according to
\[
h_-(\la,n) =h_1(\la,n) +h_2(\la,n),
\]
where
\[
h_1(\la,n)=W_{q,n-1}^+(\zeta_-(\la,\cdot),
d_-(\la,n,\cdot)), \quad h_2(\la,n)=h_-(\la,n) - h_1(\la,n),
\]
\be \label{5.55}
d_-(\la,n,.):=\int_{\sig_-^{(1),u}}
\frac{|T_-(\xi)|^2 \psi_+(\xi,n) \psi_q^+(\xi,\cdot) } {\xi-\lam
}d\omega_-(\xi).
\ee
It follows from the proof of Lemma \ref{lem5.1}
that $h_2(\la)=O(\zeta_-(\la))$ for $\la\to E$. Recall that at the
point under consideration singularities
$E\in\{\mu_1^+, \dots, \mu_{g_+}^+\} \cup \hat M_-$ might occur (in the case
$\pa\si_-^{(1)}\cap\pa\si$ one can have $E\in M_+\cup\breve M_+$ and
in the case $\pa\si_-^{(1)}\cap\pa\si_+^{(1)}$ one can have
$E\in\hat M_+$). Introduce
\beq\label{new9}
\phi_q^+(z,n):=\breve\de_+(z)\breve\psi_q^+(z,n)
\eeq
and recall that \eqref{new7} implies
\beq\label{new8}
\phi_q^+(z,n) - \phi_q^+(E,n)=O(\sqrt{z - E}).
\eeq
Then (see \eqref{defP} and \eqref{0.63}) we have
\be \label{5.56}
\frac{h_+ \zeta_-}{W}=O\bigg(\frac{h_+\breve{\psi}_q^+}{W T_+}\bigg)=
O\bigg(\frac{h_+\hat\de_+\breve\de_+\breve{\psi}_q^+}{P_+}\bigg)=
O\bigg(\frac{h_+ \hat\de_+}{P_+}\bigg)\phi_q^+.
\ee
Now we
distinguish two cases: (a) $E\in \pa\si_-^{(1)}\cap\pa\si_+^{(1)}$
and (b) $E\in\pa\si_-^{(1)}\cap\pa\si$.

Case (a). By \eqref{5.10} and \eqref{new4}  we have
\beq\label{new13}
\hat\de_+h_+=\frac{(R_-\psi_- +
\ol\psi_-)\hat\de_+}{T_-} = \frac{\hat W\breve\de_-(R_-\psi_- +
\ol\psi_-)}{P_-} =o\bigg(\frac{1}{\sqrt{\la - E}}\bigg),
\eeq
therefore
\be \label{5.80}
\frac{h_+(\la)\zeta_-(\la)}{W(\la)}=o\bigg(\frac{1}{\sqrt{\la -
E}}\bigg) \frac{\phi_q^+(\la)}{P_+(\la)}.
\ee
As a consequence
of $\frac{\phi_q^+}{P_+}=O\big(\frac{1}{\sqrt{\la - E}}\big)$ we
obtain
\be \label{5.82}
\frac{h_+h_2}{W}=o\bigg(\frac{1}{\la -E}\bigg),
 \quad E\in\pa\si_-.
\ee

Next, we have to estimate
\beq\label{new14}
\frac{h_+h_1}{W}=W_{q,n-1}^+\Big(\frac{h_+\zeta_-}{W},d_-\Big).
\eeq
By \eqref{new8} we can represent \eqref{5.80} as
\be \label{5.801}
\frac{h_+(\la)\zeta_-(\la)}{W(\la)}= o\bigg(\frac{1}{\sqrt{\la - E}}\bigg)
\bigg(\frac{\bar{\psi}_q^+(E)}{\sqrt{\la - E}} + O(1)\bigg).
\ee
Then \eqref{new14} implies
\begin{align} \label{5.83}
\begin{split}
\frac{h_+(\la,n)h_1(\la,n)}{W(\la)}&=o\bigg(\frac{1}{\sqrt{\la - E}}\bigg)\bigg(O(d_-(\la,n)) +
O(d_-(\la,n-1))\\
&\quad +\frac{O\left(W_{q,n-1}^+\left(\phi_q^+(E),
d_-(\la)\right)\right)}{\sqrt{\la-E}}\bigg).
\end{split}
\end{align}
To estimate $d_-$ in the first two summands we distinguish between
the resonance case, $E\in\si_v$, and non-resonance, $E \notin\si_v$.
First let $E\notin \si_v$, that is, $\hat W(E)\neq 0$. From \eqref{5.13} and \eqref{5.131}
we see that the integrand is bounded as $\la\to E\notin\si_v$, then $d_-(\la)=O(1)$ by \cite{Mu}.

If $E\in\si_v$, then \eqref{estim} (see also \eqref{5.13}) yields
$$
|T_-(\xi)|^2\rho_-(\xi)
\psi_+(\xi,\cdot)
\psi_q^+(\xi,\cdot)=O\left(\frac{1}{\sqrt{\xi-E}}\right)
$$
and \cite[Eq. (29.8)]{Mu} implies
\be \label{5.62}
d_-(\la)=o\left(\frac{1}{\sqrt{\la - E}}\right).
\ee
For the estimate of the last summand in \eqref{5.83} we use \eqref{5.13} and \eqref{5.55} to
represent the integrand in $W_{q,n-1}^+\left(\breve \psi_q^+(E), d_-(\la)\right)$ as
\begin{align*}
&|T_-(\xi)|^2\rho_-(\xi) \psi_+(\xi,n)
W_{q,n-1}^+\left(\psi_q^+(\xi),
\phi_q^+(E)\right)\\
&\hspace{1cm} =O\bigg(\frac{\sqrt{\xi - E}} {|\hat
W(\xi)|^2}W_{q,n-1}^+\big(\hat\psi_q^+(\xi),
\phi_q^+(E)\big)\bigg).
\end{align*}
It follows from \eqref{new7} and \eqref{new9} that
$$
W_{q,n-1}^+(\hat\psi_q^+(\xi),\phi_q^+(E))=O\big(\sqrt{\xi -
E}\big),
$$
which implies together with \eqref{estim} the boundedness of the
integrand near $E$. Thus, \beq\label{new12}
W_{q,n-1}^+\left(\phi_q^+(E), d_-(\la)\right)=O(1), \eeq and
combining \eqref{5.82}, \eqref{5.83}, \eqref{new12}, and
\eqref{5.62} finishes case (a).

Case (b). Now we do not have estimate \eqref{new4} (cf.\ {\bf III}, {\bf (b)}) at our disposal,
but we can proceed as in \eqref{5.56}, \eqref{new13} since $P_+(E)\neq 0$ and arrive at
\beq\label{new16}
\frac{h_+ \zeta_-}{W}=O(h_+\hat\de_+)=O\bigg(\frac{\hat W\breve\de_-(R_-\psi_- +
\ol\psi_-)}{P_-}\bigg)=O\bigg(\frac{\hat W}{\sqrt{\la - E}}\bigg).
\eeq
This estimate is sufficient to conclude that \eqref{5.82} is valid in case (b) as well.
For $h_1$, we use the following estimate (cf. \eqref{5.62} and \eqref{new16}) instead of \eqref{new14}:
$$
\frac{h_+h_1}{W}=O\left(\frac{h_+\zeta_-}{W}\right)O\left(d_-\right)=
O\left(\frac{\hat W}{\sqrt{\la - E}}\right)o\left(\frac{1}{\sqrt{\la - E}}\right).
$$
Combining this with \eqref{5.82} finishes case (b).

$\Omega_4$: Finally we have to show that the singularities of $G(z,n)$ at the
points of the discrete spectrum are removable. Since
$\tilde W(z)$ has simple zeros at $z=\la_k$, it suffices by \eqref{5.25} to show that
\beq \label{5.71}
\tilde h_+(\la_k,n)\tilde h_-(\la_k,n) =
\tilde\psi_-(\la_k,n)\tilde\psi_+(\la_k,n).
\eeq
By Lemma~\ref{lem5.1}, the functions $\tilde h_\mp=\de_\mp h_\mp$
given in (\ref{5.11}) are continuous at the points $\breve M_\pm$. Since
$(\de_\mp T_\pm^{-1})(\la_k)=0$ and
$(\de_\mp T_\pm^{-1}\breve\psi_q^\pm)(\la_k)=0$, only the last
summand in (\ref{5.11}) is non-zero.
We compute the limit of this summand as $\la \rightarrow \la_k$ using (\ref{2.18}),
\beq \label{5.77}
\tilde h_\mp(\la_k) =-\ga_{\pm,k} \tilde\psi_\pm(\la_k) \frac{d\tilde W(\la_k)}{d\la},
\eeq
and apply (\ref{2.11}) to obtain (\ref{5.71}).
\end{proof}

The identity $G(z,n)\equiv0$ implies
\beq\label{new17}
\psi_+(z,n) \psi_-(z,n)- h_+(z,n)h_-(z,n) \equiv 0, \quad \forall n \in \Z.
\eeq
For $z \rightarrow \infty$ we obtain by (\ref{1.12}) and (\ref{Dop})
$$
\psi_+(z,n)\psi_-(z,n)=K_+(n,n)K_-(n,n)\vprod{j=0}{n-1}\frac{a_q^+(j)}{a_q^-(j)}(1+o(1)).
$$
Formulas \eqref{h infty} and \eqref{T infty} imply
$$
h_+(z,n)h_-(z,n)=\frac{1}{T_+(\infty)^2 K_+(n,n)K_-(n,n)}
\vprod{j=0}{n-1}\frac{a_q^-(j)}{a_q^+(j)}(1+o(1))
$$
and by \eqref{new17},
$$
K_+(n,n)K_-(n,n)\vprod{j=0}{n-1}\frac{a_q^+(j)}{a_q^-(j)}=\frac{1}{T_+(\infty)}.
$$
The value on the left hand side does not depend on $n$, so using \eqref{5.1} we conclude
\be\label{new18}
a_+(n) = a_-(n) \equiv a(n), \quad \forall n \in \Z.
\ee
It remains to prove $b_+(n)=b_-(n)$.
If we eliminate the reflection coefficient $R_\pm$ from \eqref{5.10} at $n$ and
\eqref{4.26} at $n+1$ we obtain
\begin{align}
\begin{split}
&G_1(\la,n):=\frac{\psi_+(\lam,n) \psi_-(\lam,n+1)
-h_+(\lam,n+1) h_-(\lam,n)}{W(\lam)}\\
&= \rho_+(\la)
\big(\overline{h_\pm(\lam,n+1)} \psi_\pm(\lam,n)
- \overline{\psi_\pm(\lam,n)} h_\pm(\lam,n+1)\big), \quad \la\in\si^{(2),\mathrm{u,l}}.
\end{split}
\end{align}
Proceeding as for $G(\lam,n)$ in Lemma~\ref{lem5.2} we can show that that the function
$G_1(z,n)$ is holomorphic in $\C$. From \eqref{h infty}, \eqref{Dop}, \eqref{T infty},
 \eqref{1.12}, \eqref{new18}, and the Liouville theorem we conclude that
$$
\frac{\psi_+(z,n) \psi_-(z,n+1)-h_+(z,n+1) h_-(z,n)}{W(z)}= -1/a(n).
$$
We compute the asymptotics of
\[
\bar W(z,n) := a(n) \left(\psi_+(z,n) \psi_-(z,n+1) - h_+(z,n+1) h_-(z,n) \right)\\
= - W(z)
\]
as $z \rightarrow \infty$ and obtain (compare (\ref{B4jost}))
\be
0= \bar W(z,n) - \bar W(z,n-1) = (b_+(n) - b_-(n))K_+(0,0)K_-(0,0).
\ee
This implies in particular $b_+(n) = b_-(n) \equiv b(n)$, hence the proof of Theorem~\ref{theor2} is finished.

{\bf Acknowledgments.}
I.E. and J.M. gratefully acknowledge the extraordinary hospitality of the
Faculty of Mathematics of the University of Vienna during their stay in 2007,
where parts of this paper were written.

\end{document}